\documentclass[12pt]{amsart}
\usepackage{amsfonts}
\usepackage{amssymb,latexsym}
\usepackage{enumerate}
\usepackage{mathrsfs}
\makeatletter
\@namedef{subjclassname@2010}{%
  \textup{2010} Mathematics Subject Classification}
\makeatother

  [2002/01/19 v2.2g %
    AMS font definitions%
  ]
\DeclareFontFamily{U}{euf}{}
\DeclareFontShape{U}{euf}{m}{n}{%
  <5><6><7><8><9>gen*eufm%
  <10><10.95><12><14.4><17.28><20.74><24.88>eufm10%
  }{}
\DeclareFontShape{U}{euf}{b}{n}{%
  <5><6><7><8><9>gen*eufb%
  <10><10.95><12><14.4><17.28><20.74><24.88>eufb10%
  }{}

  [2002/01/19 v2.2g %
    AMS font definitions%
  ]
\DeclareFontFamily{U}{msb}{}
\DeclareFontShape{U}{msb}{m}{n}{%
  <5><6><7><8><9>gen*msbm%
  <10><10.95><12><14.4><17.28><20.74><24.88>msbm10%
  }{}

  [2002/01/19 v2.2g %
    AMS font definitions%
  ]
\DeclareFontFamily{U}{msa}{}
\DeclareFontShape{U}{msa}{m}{n}{%
  <5><6><7><8><9>gen*msam%
  <10><10.95><12><14.4><17.28><20.74><24.88>msam10%
  }{}

\newtheorem{theorem}{Theorem}[section]
\newtheorem{lemma}[theorem]{Lemma}
\newtheorem{proposition}[theorem]{Proposition}

\theoremstyle{definition}

\newtheorem{remark}[theorem]{Remark}

\numberwithin{equation}{section}
\begin{document}

\title[$(S,\{2\})$-Iwasawa theory]
{The $(S,\{2\})$-Iwasawa theory}

\author{Su Hu and Min-Soo Kim}

\address{Department of Mathematics, South China University of Technology, Guangzhou, Guangdong 510640, China}
\email{hus04@mails.tsinghua.edu.cn}

\address{Center for General Education, Kyungnam University,
7(Woryeong-dong) kyungnamdaehak-ro, Masanhappo-gu, Changwon-si, Gyeongsangnam-do 631-701, Republic of Korea
}
\email{mskim@kyungnam.ac.kr}

\subjclass[2000]{11R23, 11S40, 11S80} \keywords{$p$-adic Euler $L$-function, Iwasawa theory}

%\date{January 1, 2001 and, in revised form, June 22, 2001.}


\begin{abstract}
Iwasawa made the fundamental discovery  that there is a  close connection between  the ideal class groups of $\mathbb{Z}_{p}$-extensions of cyclotomic fields and the $p$-adic analogue of Riemann's  zeta functions
$$\zeta(s)=\sum_{n=1}^{\infty}\frac{1}{n^{s}}.$$
In this paper, we show that there may also exist  a parallel Iwasawa's  theory corresponding to the $p$-adic analogue of Euler's deformation of zeta functions
$$\phi(s)=\sum_{n=1}^{\infty}\frac{(-1)^{n-1}}{n^{s}}.$$
\end{abstract}

\maketitle

% we give a new approach to

\def\C{\mathbb C_p}
\def\BZ{\mathbb Z}
\def\Z{\mathbb Z_p}
\def\Q{\mathbb Q_p}
\def\C{\mathbb C_p}
\def\BZ{\mathbb Z}
\def\Z{\mathbb Z_p}
\def\Q{\mathbb Q_p}
\def\psum{\sideset{}{^{(p)}}\sum}
\def\pprod{\sideset{}{^{(p)}}\prod}
\def\ord{{\rm ord}}

\section{Introduction}
Throughout this paper we shall use the following notations.
\begin{equation*}
\begin{aligned}
\qquad \mathbb{C}  ~~~&- ~\textrm{the field of complex numbers}.\\
\qquad p  ~~~&- ~\textrm{an odd rational prime number}. \\
\qquad\mathbb{Z}_p  ~~~&- ~\textrm{the ring of $p$-adic integers}. \\
\qquad\mathbb{Q}_p~~~&- ~\textrm{the field of fractions of}~\mathbb Z_p.\\
\qquad\mathbb C_p ~~~&- ~\textrm{the completion of a fixed algebraic closure}~\overline{\mathbb Q}_p~ \textrm{of}~\mathbb Q_{p}.
\end{aligned}
\end{equation*}

Before Kubota, Leopoldt and Iwasawa, all the zeta functions are considered in
the complex field $\mathbb{C}$.

For Re$(s)>1$, the Riemann zeta function is defined by
\begin{equation}~\label{Riemann}
\zeta(s)=\sum_{n=1}^{\infty}\frac{1}{n^{s}}.
\end{equation}
This function can be analytically continued to a meromorphic function  in the
complex plane with a simple pole at $s=1$.

For Re$(s)>0$, the alternative series (also called the Dirichlet  eta function or Euler zeta function)
is defined by
\begin{equation}~\label{Euler}
\phi(s)=\sum_{n=1}^{\infty}\frac{(-1)^{n-1}}{n^{s}}.
\end{equation}
This function can be analytically continued  to the complex plane without any pole.

For Re$(s)>1$, (\ref{Riemann}) and (\ref{Euler}) are connected by the following equation
\begin{equation}~\label{Riemann-Euler}
\phi(s)=(1-2^{1-s})\zeta(s).
\end{equation}

According to Weil's history~\cite[p.~273--276]{Weil} (also see a survey by Goss~\cite[Section 2]{Goss}),
Euler used (\ref{Euler}) to investigate (\ref{Riemann}).
In particular, he conjectured (``proved")
\begin{equation}~\label{fe}
\frac{\phi(1-s)}{\phi(s)}=-\frac{\Gamma(s)(2^{s}-1)\textrm{cos}(\pi s/2)}{(2^{s}-1)\pi^{s}},
\end{equation}
this leads to the functional equation of $\zeta(s)$.

For $0 < x \leq 1$, Re$(s) >1$, in 1882, Hurwitz~\cite{Hurwitz}
defined the partial zeta functions
\begin{equation}~\label{Hurwitz}
\zeta(s,x)=\sum_{n=0}^{\infty}\frac{1}{(n+x)^{s}}
\end{equation}
which generalized (\ref{Riemann}).
As (\ref{Riemann}), this function can also be analytically continued to a meromorphic function  in the
complex plane with a simple pole at $s=1$.

For $0 < x \leq 1$, Re$(s) > 0$, Lerch~\cite{Ler} generalized
~(\ref{Euler}) to define the so-called Lerch zeta functions.
The following (we call it ``Hurwitz-type Euler zeta function") is a special case of Lerch's definition
\begin{equation}~\label{Hurwitz-Euler}
\zeta_{E}(s,x)=2\sum_{n=0}^{\infty}\frac{(-1)^{n}}{(n+x)^{s}}.
\end{equation}
As (\ref{Euler}), this function can be analytically continued to the complex plane without any pole.

Now we go on our story in the $p$-adic complex plane $\mathbb{C}_{p}$.

In 1964, Kubota and Leopoldt~\cite{KL} first defined the $p$-adic analogue of (\ref{Riemann}).
In fact, they defined the $p$-adic zeta functions by interpolating the special values of (\ref{Riemann}) at nonpositive integers.

In 1975, Katz \cite[Section 1]{Katz} defined the $p$-adic analogue of ~(\ref{Euler}) by interpolating the special values of (\ref{Euler}) at nonpositive integers.

In 1976, Washington~\cite{Wa2} defined the $p$-adic analogue of (\ref{Hurwitz}) for $x\in \mathbb{Q}_{p}\setminus\mathbb{Z}_{p}$, so called Hurwizt-Washinton functions (see Lang~\cite[p.~391]{Lang}). This definition has been generalized to $\mathbb{C}_{p}$ by Cohen in his book~\cite[Chapter 11]{Cohn}, and Tangedal-Young in~\cite{TP}. Both Cohen, Tangedal-Young's definitions are based on the following $p$-adic representation of Bernoulli poynomials by the  Volkenborn
integral
\begin{equation}\int_{\Z} (x+a)^n da=B_n(x),
\end{equation}
 where the Bernoulli polynomials are defined by the following generating function
 \begin{equation}\label{Bernoulli}
\frac{te^{xt}}{e^t-1}=\sum_{n=0}^\infty B_n(x)\frac{t^n}{n!},
\end{equation}
and the Volkenborn
integral of any strictly differentiable function  $f$ on $\mathbb{Z}_{p}$ is defined by
\begin{equation}\label{vol-int}
\int_{\mathbb
Z_p}f(x)dx=\lim_{N\to\infty}\frac1{p^N}\sum_{x=0}^{p^N-1}f(x)
\end{equation}
(see \cite[p.~264]{Ro}). This integral was introduced by Volkenborn
\cite{Vo} and he also investigated many important properties of
$p$-adic valued functions defined on the $p$-adic domain (see
\cite{Vo,Vo1}).

The Euler polynomials are defined by the following generating function
\begin{equation}\label{Euler-n}
\frac{2e^{xt}}{e^t+1}=\sum_{n=0}^\infty E_n(x)\frac{t^n}{n!}
\end{equation} (see \cite{Sun,MSK}). They are the special values of (\ref{Hurwitz-Euler}) at nonpositive integers (see Choi-Srivastava~\cite[p.~520, Corollary 3]{CS} and T. Kim~\cite[p.~4, (1.22)]{TK08}) and can be representative by the fermionic $p$-adic integral as follows
\begin{equation}\label{int-Em}
\int_{\Z} (x+a)^nd\mu_{-1}(a)=E_n(x)
\end{equation}
(see \cite[p.~2980, (2.6)]{KH1}),
where the fermionic $p$-adic integral $I_{-1}(f)$ on $\mathbb Z_p$ is
defined by
\begin{equation}\label{-q-e}
I_{-1}(f)=\int_{\mathbb Z_p}f(a)d\mu_{-1}(a)
=\lim_{N\rightarrow\infty}\sum_{a=0}^{p^N-1}f(a)(-1)^a
\end{equation}
(see \cite[p.~2978, (1.3)]{KH1}).

The above representation (\ref{int-Em})  and the fermionic $p$-adic integral (\ref{-q-e}) (in our natation, the $\mu_{-1}$ measure) were independently founded by Katz \cite[p.~486]{Katz} (in Katz's notation, the $\mu^{(2)}$-measure), Shiratani and
Yamamoto \cite{Shi}, Osipov \cite{Osipov}, Lang~\cite{Lang} (in Lang's notation, the $E_{1,2}$-measure), T. Kim~\cite{TK} from very different viewpoints.

Following Cohen ~\cite[Chapter 11]{Cohn} and Tangedal-Young ~\cite{TP}, using the fermionic $p$-adic integral instead of  the Volkenborn integral, we~\cite{KH1} defined $\zeta_{p,E}(s,x)$, the $p$-adic analogue of (\ref{Hurwitz-Euler}), which interpolates (\ref{Hurwitz-Euler}) at nonpositive integers (\cite[Theorem 3.8(2)]{KH1}), so called the $p$-adic Hurwitz-type Euler zeta functions. We also proved many fundamental results for the $p$-adic  Hurwitz
type Euler zeta functions, including the convergent Laurent series
expansion, the distribution formula, the functional equation, the
reflection formula, the derivative formula and the $p$-adic Raabe
formula. Using these zeta function as building blocks, we have given
a definition for the corresponding $L$-functions $L_{p,E}(\chi,s)$,
so called $p$-adic Euler $L$-functions
(in fact, this $L$-function has already founded by Katz in~\cite[p.~483]{Katz} using Kubota-Leopoldt's methords on the interpolation of $L$-functions at special values). The Hurwitz-type Euler zeta functions interpolate Euler polynomials $p$-adically (\cite[Theorem 3.8(2)]{KH1}), while the $p$-adic Euler $L$-functions
interpolate the generalized Euler numbers $p$-adically (\cite[Proposition 5.9(2)]{KH1}).

In a subsequent work~\cite{KH2},  using the fermionic $p$-adic integral, we defined the corresponding $p$-adic Diamond Log
Gamma functions. We  call them the $p$-adic Diamond-Euler Log Gamma
functions. They share most properties of the original $p$-adic Diamond Log
Gamma functions as stated in Lang's book (see \cite[p.~395--396, \textbf{G}$_{p}$\,\textbf{1}-\textbf{5} and Theorem 4.5)]{Lang}.
Furthermore, using the $p$-adic Hurwitz-type Euler zeta functions,  we  found that  the derivative of the $p$-adic Hurwitz-type Euler zeta functions
$\zeta_{p,E}(\chi,s)$ at $s=0$  may be represented by the $p$-adic  Diamond-Euler Log Gamma
functions. This led us to connect  the $p$-adic Hurwitz-type Euler zeta functions to the $(S,\{2\})$-version of the
abelian rank one Stark conjecture (see ~\cite[Chapter 6]{KH2}).

It has been pointed out that some properties for the   $q$-analogue of $p$-adic  Euler  zeta and $L$-functions have also been obtained  by T. Kim (see \cite{TK,top3,JMAA2}).

The $p$-adic zeta ($L$-) functions become central themes in algebraic number theory after Iwasawa's work. In \cite{Iw11}, Iwasawa made the fundamental discovery  that there is a close connection between his work on the ideal class groups of $\mathbb{Z}_{p}$-extensions of cyclotomic fields and the $p$-adic analogue of $L$-functions by Kubota-Leopoldt corresponding to (\ref{Riemann}).

Let $\mathbb{Q}(\mu_{p^{n+1}})$ denote the $p^{n+1}$-th cyclotomic field. In fact, Iwasawa~\cite{Iw} and Ferrero-Washington~\cite{FW} proved the following result.

\begin{theorem}[{See Lang \cite[p.~260]{Lang}}]~\label{Iwasawa} Let $h_{n}$ be the class number of $\mathbb{Q}(\mu_{p^{n+1}})$. There exist constants  $\lambda$ and $c$ such that
\begin{equation}
 \ord_{p} h_{n}^{-}=\lambda n+c.
\end{equation} for all sufficient large $n$.
\end{theorem}

Let $K$ be a number field, and choose a finite set $S$ of places $K$ containing all the archimedean places. Let $T$ be a finite set of places of $K$ disjoint from $S$. The $(S,T)$-class groups of global fields have been studied in detail by Rubin~\cite{Rubin}, Tate~\cite{Tate}, Gross~\cite{Gross}, Darmon~\cite{Darmon}, Vallieres~\cite{Va1,Va2} (we shall recall some notations on the $(S,T)$-refined class groups of global fields in the next section).  Let $K=\mathbb{Q}(\mu_{p^{n+1}})$ and $K^{+}=\mathbb{Q}(\mu_{p^{n+1}})^{+}$ be the $p^{n+1}$-th cyclotomic field and its maximal real subfield, respectively. Let $S$ be the set of infinite places of $K$, $T$  be the set of places above 2, $h_{n,2}$ and $h_{n,2}^{+}$ be the $(S,T)$-refined class numbers of $K$ and $K^{+}$ respectively (the definition will be given in the next section), and $h_{n,2}^{-}=h_{n,2}/h_{n,2}^{+}$.

Using the $p$-adic analogue of $L$-functions corresponding to Euler's deformation of zeta functions (\ref{Euler}), We shall prove the following result (comparing with  Theorem~\ref{Iwasawa}).
\begin{theorem}[$(S,\{2\})$-Iwasawa theory]~\label{main} There exist constants  $m$, $\lambda$ and $c$ such that
\begin{equation}
 \ord_{p} h_{n,2}^{-}=m p^{n} + \lambda n+ c
\end{equation} for all sufficient large $n$.
\end{theorem}

Our paper is organized as follows.

In Section 2, we shall recall some notations and results on the $(S,T)$-refined class groups of global fields. In Section 3, from the
Euler product decompositions of the $(S,T)$-Dedekind zeta functions, we shall express $h_{n,2}^{-}$ as the product of generalized Euler numbers.
 In section 4, we shall prove Theorem~\ref{main}.

\section{$(S,T)$-refined class number formula (\cite[Section 1]{Gross})}
In this section, we shall recall some notations and results on the $(S,T)$-refined class groups of global fields following very closely the expositions of Gross in \cite[Section 1]{Gross} and Aoki in \cite[Section 7]{Aoki}.

Let $k$ be a global field. Let $S$ be a finite set of places of $k$ which is nonempty and contains all archimedean places. Let $T$ be a finite set of places of $k$ which is disjoint from $S$. Let $A$ be the ring of $S$-integers and let $U_{S}=A^{*}$ be the group of $S$-units.
Let $J_{k}$ be the id\`{e}le group of $k$. If $\mathfrak{p}$ is a place of $k$, then we denote by $k_{\mathfrak{p}}$ and $A_{\mathfrak{p}}$ the completion of $k$ and $A$ at $\mathfrak{p}$ respectively, we also denote by $\mathbb{F}_{\mathfrak{p}}$ the residue field of $\mathfrak{p}$.  We define the $(S,T)-$ id\`{e}le group  $J_{S,T}$ to be the subgroup $$J_{S,T}=\prod_{\mathfrak{p}\in S}k_{\mathfrak{p}}^{*}\times\prod_{\mathfrak{p}\in T}A_{\mathfrak{p},1}^{*}\times\prod_{\mathfrak{p}\not\in S\cup T}A_{\mathfrak{p}}^{*}$$ of $J_{k}$, where for $\mathfrak{p}\in T$ we put $A_{\mathfrak{p},1}^{*}=\{u\in A_{\mathfrak{p}}^{*} ~|~ u\equiv 1 ~ (\textrm{mod}~\mathfrak{p})\}.$
Define the $(S,T)$-unit group of $k$ by $$U_{S,T}=\{u\in U_{S}~|~~ u\equiv 1 ~ (\textrm{mod}~\mathfrak{p})~\textrm{for~all~}\mathfrak{p}\in T\}.$$ Clearly we have $$U_{S,T}=k^{*}\cap J_{S,T}.$$ The $(S,T)$-id\`{e}le class group is defined to be the quotient group $$C_{S,T}=J_{S,T}/U_{S,T}.$$ Let $C_{k}=J_{k}/k^{*}$ be the id\`{e}le class group
of $k$. The $(S,T)$- ideal class group Pic$(A)_{S,T}$ of $k$ is defined by $$\textrm{Pic}(A)_{S,T}=C_{k}/C_{S,T}.$$
Denote by  $\textrm{Pic}(A)_{S}=\textrm{Pic}(A)_{S,\emptyset}$.
The class group Pic$(A)_{S}$ is finite of order $h$, and the unit group $U_{S}$ is finitely generated of
rank $n=\#S-1$. The torsion subgroup of $U_{S}$ is equal to the group of roots of unity $\mu$ in $k$; it is cyclic of order $w$.

Let $Y$ be the free abelian group generated by the places $v\in S$ and $X=\left\{\sum a_{v}\cdot v:\sum a_{v}=0\right\}$ the subgroup of elements of degree zero in $Y$.
The $S$-regulator $R$ is defined as the absolute value of the determinant of the map
\begin{equation}
\begin{aligned}
\lambda: U_{S}&\rightarrow \mathbb{R}\otimes X\\
\epsilon &\mapsto\sum_{S} \log\|\epsilon\|_{v}\cdot v,
\end{aligned}
\end{equation}
taken with respect to $\mathbb{Z}$-bases of the free abelian groups $U_{S}/\mu_{S}$ and $X$.

The zeta function of $A$ is given by
\begin{equation}
\zeta_{S}(s)=\prod_{\mathfrak{p}\not\in S}\frac{1}{1-N\mathfrak{p}^{-s}}
\end{equation}
in the half plane Re($s)>1.$ It has a meromorphic continuation to the $s$-plane, with a simple pole at $s=1$ and no other singularities. At $s=0$ the Taylor
expansion begins:
\begin{equation}~\label{1.3}
\zeta_{S}(s)\equiv\frac{-hR}{w}\cdot s^{n} \pmod{s^{n+1}}
\end{equation}
(see \cite[p.~178, (1.3)]{Gross}).

Let $T$ be a finite set of places of $k$ which is disjoint from $S$, and define
\begin{equation}~\label{R-zeta}
\zeta_{S,T}(s)=\prod_{\mathfrak{p}\in T}(1-N\mathfrak{p}^{1-s})\cdot\zeta_{S}(s),
\end{equation}
we shall call it the $(S,T)$-refined zeta function of $k$ throughout this paper.
From the discussions of Aoki in \cite[p.~471--472]{Aoki}, we have an exact sequence
\begin{equation}~\label{exact}
1\rightarrow U_{S,T}\rightarrow U_{S} \rightarrow \prod_{\mathfrak{p}\in T}\mathbb{F}_{\mathfrak{p}}^{*} \rightarrow {\rm Pic}(A)_{S,T}
\rightarrow {\rm Pic}(A)_{S} \rightarrow 1.
\end{equation}
Let  $h_{S,T}$ be the order of Pic$(A)_{S,T}$ (we call it the $(S,T)$-refined  class number throughout this paper), $R_{S,T}$ be the determinant
of $\lambda$ with respect to basis of $U_{S,T}/\mu_{S,T}$ and $X$, and $w_{S,T}$ be the order of roots of unity $\mu_{S,T}$ which
are $\equiv 1$ (mod $T$), we have the following $(S,T)$-refined class number formula due to Gross
\begin{equation}~\label{R-cnf}
\zeta_{S,T}(s)\equiv\frac{-h_{S,T}R_{S,T}}{w_{S,T}}\cdot s^{n}\pmod{s^{n+1}}
\end{equation}(see \cite[p.~179, (1.6)]{Gross}).

\section{Refined class number and the generalized Euler numbers}
Let $K=\mathbb{Q}(\mu_{p^{n+1}})$ and $K^{+}=\mathbb{Q}(\mu_{p^{n+1}})^{+}$ be the $p^{n+1}$-th cyclotomic field
and its maximal real subfield, respectively. Let $S$ be the set of infinite places of $K$, $T$  be set of the places above 2, $h_{n,2}, h_{n,2}^{+}, U_{n,2}, U_{n,2}^{+}, $  $\mu_{n,2},
\mu_{n,2}^{+}, w_{n,2}, w_{n,2}^{+}, R_{n,2}, R_{n,2}^{+}$ denote all the quantities or objects of $K$ and $K^{+}$
which are refined by $T$ as in the above
Section. Let $\zeta_{K,2}(s)$ be the $(S,T)$-zeta function of $K$ (see (\ref{R-zeta})),
and  \begin{equation}\label{Le}L_{E}(s,\chi)=2\sum_{n=1}^{\infty}\frac{(-1)^{n}\chi(n)}{n^{s}},~~\textrm{Re}(s)>0\end{equation}
 be the Dirichlet $L$-function corresponding to ~(\ref{Euler}) (we call them the Euler $L$-functions throughout this paper).
 This function has close connection with the generalized Euler numbers and it can be continued to the entire complex plane. In \cite[Scetion 5.3]{KH1}, using formal power series expansions, we  recalled  the definition and some results on generalized Euler numbers. The Propositions 5.2 and 5.3 of \cite{KH1}
 correspond to properties (4) and (5) of the generalized Bernoulli numbers in Iwasawa's book \cite[p.~10--11]{Iw}
 (for details we also refer to ~\cite[Sections 1 and 2]{Kim}). Proposition~\ref{specialvalues} below shows that the special values of Euler $L$-functions at
 non-positive integers are the generalized Euler numbers. This is similar with a result in Iwasawa's book \cite[p.~11, Theorem 1]{Iw}
which shows that the special values of Dirichlet $L$-functions at
 non-positive integers are  the generalized Bernoulli numbers.

We have the following decomposition of $(S,\{2\})$-refined Dedekind zeta functions as the Euler $L$-functions (comparing with the last formula on \cite[p.~75]{Lang}).

\begin{proposition}~\label{decomposition}
\begin{equation}\zeta_{S,2}(s)= \prod_{\chi}\frac{1}{2}L_{E}(s,\chi),
\end{equation}
where the product is taken over all the primitive characters induced by the characters of ${\rm Gal}(K/\mathbb{Q}).$
\end{proposition}
\begin{proof}
From the last formula on \cite[p.~75]{Lang}, we have \begin{equation}~\label{dz}\zeta_{K}(s)=\prod_{\chi}L(s,\chi).\end{equation}
By (\ref{R-zeta}), we have
\begin{equation}~\label{d2} \zeta_{S,2}(s)=\prod_{\mathfrak{p}\in T}(1-N\mathfrak{p}^{1-s})\cdot\zeta_{K}(s).\end{equation}
For any Dirichlet character $\chi$ of Gal($K/\mathbb{Q}$),
\begin{equation}~\label{d3}L(s,\chi)=\prod\left(1-\frac{\chi(q)}{q^{s}}\right)^{-1}=\sum_{n=1}^{\infty}\frac{\chi(n)}{n^{s}},\end{equation}
where the product is taken over all primes $q$ such that $(q,p)=1$
(\cite[p.~76]{Lang}).
Let $$(2)=(\mathfrak{p}_{1}\cdots\mathfrak{p}_{r})^{e},~~~~N\mathfrak{p}=2^{f}$$
be the decomposition of $2$ in prime ideals in $K$. Then
$$efr=[K:\mathbb{Q}].$$
By the following identity in \cite[p.~76]{Lang}: $$(1-t^{f})^{r}=\prod_{\chi}(1-\chi(p)t),$$ we have \begin{equation}~\label{d4}\begin{aligned}\prod_{\mathfrak{p}\in T}(1-N\mathfrak{p}^{1-s})&=(1-2^{(1-s)f})^{r} \\&=\prod_{\chi}(1-\chi(2)2^{1-s}).\end{aligned}\end{equation}
Combine ~(\ref{d2}), (\ref{d3}) and (\ref{d4}), we have
\begin{equation}
\begin{aligned}
\zeta_{S,2}(s)&=\prod_{\mathfrak{p}\in T}(1-N\mathfrak{p}^{1-s})\cdot\zeta_{K}(s)\\
&=\prod_{\chi}\left[(1-\chi(2)2^{1-s})L(s,\chi)\right]\\
&=\prod_{\chi}\left(\sum_{n=1}^{\infty}\frac{\chi(n)}{n^{s}}-2\sum_{n=1}^{\infty}\frac{\chi(2n)}{(2n)^{s}}\right)\\
&=\prod_{\chi}\sum_{n=1}^{\infty}\frac{(-1)^{n-1}\chi(n)}{n^{s}}\\
&=(-1)^{\varphi(p^{n+1})}\prod_{\chi}\frac{1}{2}\cdot 2\sum_{n=1}^{\infty}\frac{(-1)^{n}\chi(n)}{n^{s}}\\
&=\prod_{\chi}\frac{1}{2}L_{E}(s,\chi).
\end{aligned}
\end{equation}
This completes the proof of our assertion.
\end{proof}

For the $(S,\{2\})$-refined zeta function of $K^{+}$, we have the following decomposition.

\begin{proposition}~\label{decomposition2}
\begin{equation}\zeta_{K^{+},2}(s)= (-1)^{\frac{\varphi(p^{n+1})}{2}}\prod_{\chi~~{\rm even}}\frac{1}{2}L_{E}(s,\chi).
\end{equation}
 \end{proposition}

The following result shows that the special values of Euler $L$-functions at
 non-positive integers are the generalized Euler numbers. It is well-known, but may be not easy to find a reference.
 So we add a proof for the completeness.

\begin{proposition}~\label{specialvalues}
For any integers $n\geq0$ we have $L_E(-n,\chi)=E_{n,\chi}.$ In particular, $L_E(0,\chi)=E_{0,\chi}.$\end{proposition}
\begin{proof}Consider the following generating function
\begin{equation}\label{def-E}
F(t,x)=\frac{2e^{xt}}{e^t+1}.
\end{equation}
Expand $F(t,x)$ into a power series of $t:$
\begin{equation}\label{def-E-pow}
F(t,x)=\sum_{n=0}^\infty E_{n}(x)\frac{t^n}{n!}.
\end{equation}
Recall that the coefficients $E_n(x),n\geq0,$ are called Euler polynomials (see (\ref{Euler-n}) above).
For a primitive Dirichlet character $\chi$ with an odd conductor $f=f_\chi,$
the formal power series $F_\chi(t)$ are defined by
\begin{equation}\label{gen-E-numb}
F_\chi(t)=2\sum_{a=1}^{f}\frac{(-1)^a\chi(a)e^{at}}{e^{ft}+1},\quad |t|<\pi/f.
\end{equation}
The generalized Euler numbers $E_{n,\chi}$ which belong to the Dirichlet character $\chi$ are defined by
\begin{equation}\label{gen-E-numb-d}
F_\chi(t)=\sum_{n=0}^{\infty}E_{n,\chi}\frac{t^n}{n!}.
\end{equation}
Let $\mathbb Q(\chi)$ denote the field generated over $\mathbb Q$ by all the values $\chi(a),a\in\mathbb Z.$
Then it can be shown that $E_{n,\chi}\in\mathbb Q(\chi)$ for each $n\geq0.$
 From (\ref{gen-E-numb}), we obtain the following
generating function of $E_{n,\chi}$ by working formally with power series:
\begin{equation}\label{power-sum}
\begin{aligned}
F_\chi(t)&=2\sum_{a=1}^{f}(-1)^a\chi(a)\sum_{j=0}^\infty(-1)^je^{(a+fj)t}\\
&=2\sum_{j=0}^\infty\sum_{a=1}^f (-1)^{a+fj}\chi(a+fj)e^{(a+fj)t} \\
&\quad\text{(by using $f$ is an odd conductor of $\chi$)} \\
&=2\sum_{l=1}^\infty (-1)^l\chi(l)e^{lt} \\
&=2\sum_{l=1}^\infty (-1)^l\chi(l)\sum_{n=0}^\infty l^n\frac{t^n}{n!} \\
&=\sum_{n=0}^\infty\left(2\sum_{l=1}^\infty(-1)^l\chi(l)l^n\right)
\frac{t^n}{n!}.
\end{aligned}
\end{equation}
Comparing coefficients of ${t^n}/{n!}$ on both sides of (\ref{gen-E-numb-d}) and (\ref{power-sum}), we have
\begin{equation}\label{power-sum-ell}
E_{n,\chi}=2\sum_{l=1}^\infty(-1)^l\chi(l)l^n
\end{equation}
(also see \cite[Theorem 7]{TK}).
From (\ref{def-E}), (\ref{def-E-pow}) and (\ref{gen-E-numb}), we also have
\begin{equation}\label{g-ep}
\begin{aligned}
F_\chi(t)&=\sum_{a=1}^{f}(-1)^a\chi(a)\frac{2e^{at}}{e^{ft}+1} \\
&=\sum_{a=1}^{f}(-1)^a\chi(a)F\left(ft,\frac af\right) \\
&=\sum_{a=1}^{f}(-1)^a\chi(a)\sum_{n=0}^\infty E_n\left(\frac af\right)\frac{(ft)^n}{n!} \\
&=\sum_{n=0}^\infty\left( f^n\sum_{a=1}^{f}(-1)^a\chi(a)E_{n}\left(\frac{a}f\right) \right)\frac{t^n}{n!}.
\end{aligned}
\end{equation}
Thus, comparing (\ref{gen-E-numb-d}) with (\ref{g-ep}), we have
\begin{equation}\label{l-e-n}
E_{n,\chi}=f^n\sum_{a=1}^{f}(-1)^a\chi(a)E_{n}\left(\frac{a}f\right).
\end{equation}
In particular, $E_{0,\chi}=\sum_{a=1}^{f}(-1)^a\chi(a)$ for all $\chi.$
Let $\chi$ be a primitive Dirichlet character with an odd conductor $f.$
Recall that the Euler $L$-function attached to $\chi$ is defined by
\begin{equation}\label{ell-ft}
L_E(s,\chi)=2\sum_{n=1}^\infty\frac{(-1)^n\chi(n)}{n^s},
\end{equation}
where $s\in\mathbb C$ with Re$(s)>0$ (see (\ref{Le}) above).
The Euler $L$-function attached to $\chi$ can be continued to the entire complex plane.
From (\ref{power-sum}) and (\ref{ell-ft}) we can deduce the formula
\begin{equation}\label{sum-ell}
2\sum_{n=1}^\infty(-1)^n\chi(n)e^{-nt}=\sum_{j=0}^\infty L_E(-j,\chi)\frac{(-t)^j}{j!}
\end{equation}
by considering the vertical line integral
$\frac1{2\pi i}\int_{2-i\infty}^{2+i\infty}t^{-s}\Gamma(s)L_E(s,\chi)ds$
and moving the path integration to Re$(s)=-\infty$.

Indeed (see \cite[p.~149, line 10--18]{MR}), recall that $\Gamma(s)$ has simple poles at $s=-j$ with $j$ non-negative integer and
the residue of $\Gamma(s)$ at $s=-j$ is $(-1)^j/j!.$ Using this, as well as the calculus of residues, one derives
the familiar formula
\begin{equation}\label{n-exp}
e^{-t}=\frac1{2\pi i}\int_{2-i\infty}^{2+i\infty}t^{-s}\Gamma(s)ds.
\end{equation}
Thus, if
\begin{equation}\label{n-exp-di}
F(s)=\sum_{n=1}^\infty\frac{a_n}{n^s}
\end{equation}
is a Dirichlet series convergent in Re$(s)>1,$ then we have
\begin{equation}\label{n-exp-di-c}
\sum_{n=1}^\infty a_n e^{-nt}=\frac1{2\pi i}\int_{2-i\infty}^{2+i\infty}t^{-s}\Gamma(s)F(s)ds.
\end{equation}
By setting $F(s)=L_E(s,\chi)$ in (\ref{n-exp-di-c}), we have
\begin{equation}\label{e-z-ne}
\begin{aligned}
2\sum_{n=1}^\infty (-1)^n\chi(n)e^{-nt}=\frac1{2\pi i}\int_{2-i\infty}^{2+i\infty}t^{-s}\Gamma(s)L_E(s,\chi)ds.
\end{aligned}
\end{equation}
Moving the line of integration to the left, and picking up the contribution from the poles of $\Gamma(s),$ we conclude
that the right hand side of the above equation is
\begin{equation}\label{e-z-ne-r}
\sum_{j=0}^\infty \frac{(-1)^jL_E(-j,\chi)t^j}{j!}
\end{equation}
(also see \cite[(3.11)]{MR}). Thus we get (\ref{sum-ell}).

Otherwise, the left hand side of (\ref{sum-ell}) can be simplified as follows:
\begin{equation}\label{sum-eu}
\begin{aligned}
2\sum_{n=1}^\infty(-1)^n\chi(n)e^{-nt}
&=2\sum_{a=1}^f\sum_{j=0}^\infty(-1)^{a+fj}\chi(a+fj)e^{-(a+fj)t} \\
&\quad\text{(we write $n=fj+a$, where $1\leq a\leq f$ and $j=0,1,\ldots$)} \\
&=2\sum_{a=1}^{f}(-1)^a\chi(a)e^{-at}\sum_{j=0}^\infty(-1)^je^{-fjt}\\
&\quad\text{(since $f$ is odd)} \\
&=\sum_{a=1}^{f}(-1)^a\chi(a)\frac{2e^{(1-a/f)(ft)}}{e^{ft}+1} \\
&=\sum_{a=1}^{f}(-1)^a\chi(a)F(ft,(1-a/f)) \\
&\quad\text{(by (\ref{def-E}))} \\
&=\sum_{j=0}^\infty\left(f^j\sum_{a=1}^f(-1)^a\chi(a)E_j\left(1-\frac af\right)\right)\frac{t^j}{j!} \\
&\quad\text{(by (\ref{def-E-pow}))}.
\end{aligned}
\end{equation}
Comparing  (\ref{sum-ell}) with (\ref{sum-eu}), we have
\begin{equation}\label{proof} L_{E}(-j,\chi)= f^j\sum_{a=1}^f(-1)^{a+j}\chi(a)E_j\left(1-\frac af\right)\end{equation}
for $j\geq0.$

Thus by (\ref{proof}) and the identity $E_{j}(1-a/f)=(-1)^jE_j(a/f)$, we have
\begin{equation}\label{ell-eu-ne}
L_E(-j,\chi)=f^j\sum_{a=1}^{f}(-1)^a\chi(a)E_{j}\left(\frac{a}f\right).
\end{equation}
Finally, by (\ref{l-e-n}), we can rewrite (\ref{ell-eu-ne}) as
$$L_E(-n,\chi)=E_{n,\chi}$$
for $n\geq0$. In particular, we have $L_E(0,\chi)=E_{0,\chi}.$
\end{proof}

By Propositions \ref{decomposition2} and \ref{specialvalues},  we can represent $h_{n,2}^{-}$ by the product of generalized Euler numbers as follows  (comparing with  \cite[Theorem 3.2]{Lang}).

 \begin{proposition}~\label{product}
 \begin{equation} h_{n,2}^{-}= (-1)^{\frac{\varphi(p^{n+1})}{2}}2^{1-\varphi(p^{n+1})}
 \prod_{\chi~~{\rm odd}}E_{0,\chi},
 \end{equation}
 where $E_{0,\chi}$ are the generalized Euler numbers {\rm(\cite[Section 5.1]{KH1})}.
 \end{proposition}
 \begin{remark} From this, we also see that $E_{0,\chi}\not=0$, when $\chi$ is an odd character. In fact, $E_{0,\chi}\not=0$ if and only if $\chi$ is an odd character by \cite[Proposition 5.1]{KH1}, this phenomenon is different from the generalized Bernoulli number $B_{0,\chi}$, since $B_{0,\chi}$=0, for $\chi\not=\chi_{0}$, but corresponds to $B_{1,\chi}$, for details, we refer to \cite[p.~13, ii)]{Iw}.
 \end{remark}
 \begin{proof}
By the exact sequence (\ref{exact}) above and $h_{n}^{+}\mid h_{n}$, we know that $h_{n,2}^{+}\mid h_{n,2}$.

By Propositions~\ref{decomposition} and \ref{decomposition2}, we have
 \begin{equation}\label{div}
 \frac{\zeta_{S,2}(s)}{\zeta_{K^{+},2}(s)}=(-1)^{\frac{\varphi(p^{n+1})}{2}} \prod_{\chi~~\textrm{odd}}\frac{1}{2}L_{E}(s,\chi).
 \end{equation}
 From the $(S,T)$-refined class number formula (\ref{R-cnf}) and (\ref{div}), we have
 \begin{equation}~\label{div2}
 \begin{aligned}
  \frac{h_{n,2}R_{n,2}}{w_{n,2}}\biggl/\frac{h_{n,2}^{+}R_{n,2}^{+}}{w_{n,2}^{+}}&=\lim_{s\rightarrow 0}\frac{\zeta_{S,2}(s)}{\zeta_{K^{+},2}(s)}\\
 &=(-1)^{\frac{\varphi(p^{n+1})}{2}} \prod_{\chi~~\textrm{odd}}\frac{1}{2}L_{E}(0,\chi).
 \end{aligned}
 \end{equation}
 By Corollary 4.13 and Lemma 4.15 of \cite{Wa}, we have $R_{n,2}/R_{n,2}^{+}=2^{\frac{\varphi(p^{n+1})}{2}-1}$.
 It also easy to see $\mu_{n,2}=\mu_{n,2}^{+}=\langle-1\rangle.$ By Proposition~\ref{specialvalues}, we have $L_{E}(0,\chi)=E_{0,\chi}$.
 Thus by~(\ref{div2}), we have $$\frac{h_{n,2}}{h_{n,2}^{+}}= (-1)^{\frac{\varphi(p^{n+1})}{2}}2^{1-\varphi(p^{n+1})}
 \prod_{\chi~~\textrm{odd}}E_{0,\chi}.$$
 This implies our result.
 \end{proof}

 \section{Proof of the main result}
Let $\chi$ be a Dirichlet character modulo $p^{v}$ for some $v$.
For any $a\in\mathbb{Z}_{p}$, we set $\chi(a) := \chi(a~~ \textrm{mod}~~
p^{v})$, then $\chi$ becomes
a Dirichlet character on $\mathbb{Z}_{p}$ .

 Suppose  that the conductor of $\chi$ equals to $p^{n}$ and  $\zeta$ is a primitive $p^{n}$th root of unity.
  As in Lang \cite[p.~248]{Lang}, for any $p$-adic measure $\mu$ on $\mathbb{Z}_{p}$,  let $$B(\chi,\mu)=\int_{\mathbb{Z}_{p}}\chi(x) d\mu(x)=f(\zeta-1),$$ where $f$ is the power series associated with $\mu$. Let $\mathfrak{o}$
 be the ring of $p$-adic integers in $\mathbb{C}_{p}$.
 Suppose that there exists a rational number $m$ such that the power series $f$ can be written in the form $$f(X)=p^{m}(c_{0}+c_{1}X+\cdots+c_{\lambda-1}X^{\lambda-1}+c_{\lambda}X^{\lambda}+\cdots)$$
 where $c_{\lambda}$ is a unit in $\mathfrak{o}$, and $c_{0},\ldots,c_{\lambda-1}\in\mathfrak{m}$, the maximal ideal of $\mathfrak{o}$. We call $m,\lambda$ the Iwasawa invariants of $\mu$, or $f.$ Denote by $$x\sim y$$ if $x,y$ have the same order at $p$.
\begin{lemma}[{See Lang \cite[p.~248, Corollary 2]{Lang}}]\label{Lang3}There exists a positive integer $n_{0}$
 such that if $n\geq n_{0}$ and {\rm Cond}~$\chi=p^{n}$, then
 $$B(\chi,\mu) \sim p^{m}(\zeta -1 )^{\lambda},$$
where $\zeta$ is a primitive $p^{n}$th root of unity.
 \end{lemma}

 \begin{lemma}[{See Lang \cite[p.~249, Corollary 3]{Lang}}]\label{Lang2} For some constant $c$, we have
 $$\ord_{p}\prod_{\substack{{\rm Cond}~\chi=p^{t}\\n_0\leq t\leq n}}B(\chi,\mu)=mp^{n}+\lambda n+c.$$
 \end{lemma}

For $n\geq 0$, $E_{n,\chi}$  be the  generalized Euler numbers which was defined in \cite[Section 5.1]{KH1} and (\ref{gen-E-numb-d}). By \cite[Proposition 5.4(2)]{KH1}, we have $$E_{0,\chi}=B(\chi, \mu_{-1})=\int_{\mathbb{Z}_{p}}\chi(x) d\mu_{-1}(x).$$

From the above lemmas, we have the following results.

 \begin{proposition}\label{Lang3} There exists a positive integer $n_{0}$
 such that if $n\geq n_{0}$ and {\rm Cond}~$\chi=p^{n}$, then
 $$E_{0,\chi}\sim p^{m}(\zeta -1 )^{\lambda},$$
where $\zeta$ is a primitive $p^{n}$th root of unity. \end{proposition}

 \begin{proposition}\label{Lang4} For some constant $c$, we have
 $$\ord_{p}\prod_{\substack{{\rm Cond}~\chi=p^{t}\\n_0\leq t\leq n}}E_{0,\chi}=mp^{n}+\lambda n+c.$$
 \end{proposition}
 \begin{remark} The following observation is pointed out by the referee. We have the exact sequence
$$
1\rightarrow U_{S,T}\rightarrow U_{S} \rightarrow \prod_{\mathfrak{p}\mid 2}\mathbb{F}_{\mathfrak{p}}^{*} \rightarrow {\rm Pic}(A)_{S,T}
\rightarrow {\rm Pic}(A)_{S} \rightarrow 1.
$$
The $p$-part of ${\rm Pic}(A)_{S}$ grows like $\lambda n+v$ for integers $\lambda, v$. And the $p$-part of
$\prod_{\mathfrak{p}\mid 2}\mathbb{F}_{\mathfrak{p}}^{*}$ grows
like the $p$-part of $(2^{f} - 1)^r$ where $f$ is the residue above 2 and $r$ is the splitting. Since
$fr = n$ and $p > 2$, it seems like this implies $m = 0$. This is a little unsatisfactory since
we're strongly using the smallness of 2.\end{remark}

Finally, by Propositions~\ref{product} and ~\ref{Lang4}, we obtain Theorem~\ref{main}.

\section*{Acknowledgment}
We thank the referee for his/her many constructive comments and remarks which have  improved the quality of presentation.

The first author is  supported by the Fundamental Research Funds for the Central Universities and the  Faculty Scientific Research Foundation of South China University of Technology. The second author is  supported by the Kyungnam University Foundation Grant, 2015.

\end{document}